\documentclass[letterpaper,10pt,pdflatex,reqno]{amsart}
\usepackage{amsmath,times} \usepackage{amsthm}

\usepackage{graphicx}
\usepackage{color}

\usepackage{amssymb} \usepackage{latexsym} \usepackage{amscd}
\usepackage{verbatim} 
\newtheorem{theorem}{Theorem}[section]
\newtheorem{theorem*}{Theorem}
\newtheorem{lemma}[theorem]{Lemma}
\newtheorem{proposition}[theorem]{Proposition}
\newtheorem{corollary}[theorem]{Corollary}

\theoremstyle{definition} 
\newtheorem{definition}[theorem]{Definition}
\newtheorem*{definition*}{Definition}

\newtheorem{remark}[theorem]{Remark}
\newtheorem*{remark*}{Remark}

\newcommand{\Area}{\operatorname{Area}}

\begin{document}

\begin{abstract}
We study regularity properties of solutions to the Dirichlet problem for the complex Homogeneous Monge-Amp\`ere equation.   We show that for certain boundary data on $\mathbb P^1$ the solution $\Phi$ to this Dirichlet problem is connected via a Legendre transform to an associated flow in the complex plane called the Hele-Shaw flow. Using this we determine precisely the harmonic discs associated to $\Phi$.  We then give examples for which these discs are not dense in the product, and also prove that this situation persists after small perturbations of the boundary data.
\end{abstract}

\title[Harmonic Discs of the HMAE]{Harmonic Discs of Solutions to the Complex Homogeneous Monge-Amp\`ere Equation}

\thanks{During this work JR was supported by an EPSRC Career Acceleration Fellowship (EP/J002062/1). DWN has received funding from the People Programme (Marie Curie Actions) of the European Union's Seventh Framework Programme (FP7/2007-2013) under REA grant agreement no 329070}
 
\author{Julius  Ross and David Witt Nystr\"om}
\maketitle
\vspace{-5mm}
Let $(X,\omega)$ be a compact K\"ahler manifold of dimension $n$ and $\mathbb D\subset \mathbb C$ be the open unit disc.  Consider boundary data consisting of a family $\omega+dd^c\phi(\cdot,\tau)$ of K\"ahler forms where $\phi(\cdot,\tau)$ is a smooth function on $X$ for $\tau\in \partial \mathbb D$.  The Dirichlet problem for the complex Homogeneous Monge Amp\`ere equation (HMAE) asks for a function $\Phi$ on $X\times \overline{\mathbb D}$ such that
\begin{align*}
\Phi(\cdot,\tau) = \phi(\cdot, \tau) &\text{ for }\tau\in \partial \mathbb D,\\
\pi^*\omega + dd^c\Phi&\ge 0, \\
(\pi^* \omega + dd^c\Phi)^{n+1}&=0.
\end{align*}

We say $\Phi$ is a \emph{regular solution} if it is smooth and $\omega + dd^c\Phi(\cdot,\tau)$ is a K\"ahler form for all $\tau\in \overline{\mathbb D}$. By an example of Donaldson \cite{Donaldson} we know there exist smooth boundary data for which there does not exist a regular solution. Nevertheless, the equation always has a unique weak solution, which by the work of Chen \cite{Chen} with complements by B\l ocki \cite{Blocki} we know is at least ``almost'' $C^{1,1}$ (so in particular $C^{1,\alpha}$ for any $\alpha<1$). See \cite{Guedjbook} for a recent survey. 

A more subtle aspect of the regularity of solutions to the HMAE is the question of existence and distribution of harmonic discs.
    
\begin{definition*}
   Let $g\colon \mathbb D\to X$ be holomorphic.  We say that the graph of $g$ is a \emph{harmonic disc} (with respect to $\Phi$) if $\Phi$ is $\pi^*\omega$-harmonic (i.e. $\pi^*\omega+dd^c\Phi$ vanishes) along this graph.
\end{definition*}
 As is well known, a regular solution to the HMAE yields a complex foliation of $X\times \overline{\mathbb D}$ whose leaves restrict to harmonic discs in $X\times \mathbb D$.    Even when the solution is not regular, the existence of such harmonic discs is important; for instance along such a harmonic disc the density of the varying measure $\omega_{\phi(\cdot,\tau)}^n$ is essentially log-subharmonic (see \cite{BedfordBurns} \cite{ChenTian} \cite[Sec 3.2]{BoBerman}).

 It was hoped that any weak solution would enjoy a weaker form of regularity, so that a dense open subset of $X\times \mathbb D$ would be foliated by harmonic disc, but as we will see this is not always the case. \medskip

This paper describes a correspondence between on the one hand the HMAE when $X=\mathbb P^1$ and the boundary data has a certain kind of symmetry and on the other hand the so-called ``Hele-Shaw'' flow in the plane. As a result we see that  the set of harmonic discs is determined by the topology of the flow.

To state precise results, let $\omega_{FS}$ denote the Fubini-Study form on $\mathbb{P}^1$ and $\phi$ be a smooth K\"ahler potential, i.e.\ a smooth function on $\mathbb{P}^1$ such that $\omega_{FS}+dd^c\phi$ is K\"ahler. Let $\rho$ denote the usual $\mathbb{C}^{\times}$-action on $\mathbb{P}^1$ which acts by multiplication on $\mathbb{C}\subset \mathbb{P}^1$.  We consider the function $\phi(z,\tau):=\phi(\rho(\tau)z)$ as boundary data over $\mathbb P^1\times \partial \mathbb D$, so for each $\tau\in \partial\mathbb D$ we have a K\"ahler form $\omega_{FS} + dd^c \phi(\cdot,\tau)$.  We show that the solution $\Phi$ to the Homogeneous Monge-Amp\`ere equation (see \ref{def:weakhmae} for the definition) with this boundary data is intimately connected to the Hele-Shaw flow
$$\Omega_t: = \{ z : \psi_{t}(z) < \phi(z)\}$$ 
where
$$\psi_{t} := \sup\{\psi: \psi \text{ is usc and }\psi\le \phi \text{ and } \omega_{FS} + dd^c\psi\geq 0 \text{ and } \nu_{0}(\psi)\ge t\}.$$
By this we mean the supremum is over all upper semicontinuous (usc) functions from $\mathbb P^1$ to  $\mathbb R\cup \{-\infty\}$ with these properties, and $\nu_0(\psi)$ denotes the order of the logarithmic singularity (Lelong number) of $\psi$ at $0\in \mathbb C\subset \mathbb P^1$. In fact we show that the solution $\Phi$ and the family $\psi_t$ are related via a Legendre transform. \medskip

 Using this we prove the following:
\begin{theorem*}
 Let $\Phi$ be the solution to the HMAE with boundary data $\phi$ and $g\colon \mathbb D\to \mathbb P^1$ be holomorphic.   Then the graph of $g$ is a harmonic disc of $\Phi$ if and only if either
  \begin{enumerate}
  \item $g\equiv 0$, or
  \item $g(\tau) = \tau^{-1} z$ for some fixed $z\in \mathbb P^1\setminus \Omega_1$, or
  \item $\tau \mapsto \tau g(\tau)$ is a Riemann mapping for a simply connected Hele-Shaw domain $\Omega_t$ that maps $0\in \mathbb D$ to $0\in \Omega_t$.
  \end{enumerate}
\end{theorem*}

\begin{remark*}
In \cite{RossNystromHeleShaw} the authors prove that $\Omega_t$ is simply connected for $0<t\ll 1,$ so there is always an infinite number of harmonic discs of the form (3).
\end{remark*}

The Hele-Shaw flow $\Omega_t$ has a physical interpretation as describing the expansion of a liquid in a medium with permeability inversely proportional to $\Delta(\phi+\ln (1+|z|^2))$.    Guided by this one can rather easily find potentials $\phi$ for which at some time $t$ the flow domain $\Omega_t$ becomes multiply connected as in Figure 1. 

\begin{figure}[htb]\label{fig:multiple}
\includegraphics[width=0.8\textwidth]{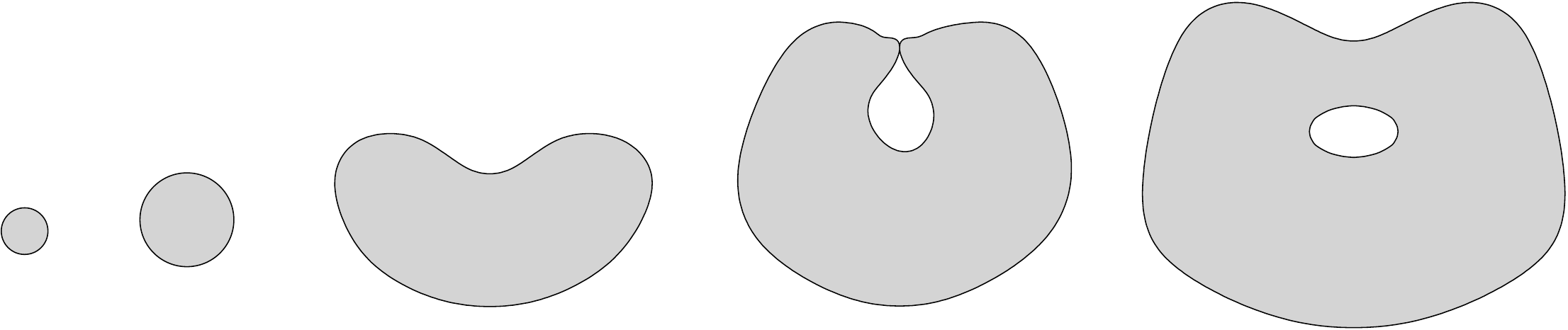}
\centering
\caption{The Hele-Shaw flow developing a multiply connected domain}  
\end{figure}

 This then translates into an obstruction to the presence of harmonic discs of the associated solution to the HMAE:

\begin{theorem*}
  There exist smooth boundary data $\phi(\cdot,\tau)$ for which the solution to the Dirichlet problem for the HMAE has the following property: there exists an open set $U$ in $\mathbb P^1\times \overline{\mathbb D}$ meeting $\mathbb P^1\times \partial \mathbb D$, such that no harmonic disc intersects $U$.
\end{theorem*}

Next we address the question whether generic boundary data give rise to  solutions with a weak form of regularity. The following theorem answers that question negatively.

\begin{theorem*}
 There exist smooth boundary data $\phi(\cdot,\tau)$ for which the following is true: there exist a nonempty open set $U'$ in $\mathbb{P}^1\times \mathbb{D}$ and an $\epsilon>0$ such that if $\phi'(z,\tau)$ is any smooth boundary data with $$\|\phi'-\phi\|_{C^2(\mathbb{P}^1\times \partial \mathbb{D})}<\epsilon$$ and $\Phi'$ is the associated solution to the HMAE then no harmonic disc (associated to $\Phi'$) passes through $U'$.
\end{theorem*}

The first theorem gives solutions that are not ``partially smooth'' and the second examples whose perturbed solutions are not ``almost smooth'',  in apparent contradiction with  \cite[Thm. 1.3.2]{ChenTian} and \cite[Thm. 1.3.4]{ChenTian} respectively (see Section \ref{sec:ChenTian}).

\subsection*{Comparison with other works}

In this paper we consider the Hele-Shaw flow with varying permeability starting from the origin which has been considered before by Hedenmalm-Shimorin; in fact the basic properties of the flow that we use are a small variant of those in \cite{Hedenmalm}. There is a much larger literature on the Hele-Shaw flow (which usually considers the case of constant permeability but with non-trivial initial condition $\Omega_0$) for which we refer the reader to the book \cite{Gustafsson} of Gustafsson-Vasil'ev and the references therein.

In \cite{Donaldson} Donaldson gave the first example of boundary data for which the Dirichlet problem for the HMAE on some $X\times \overline{\mathbb{D}}$ has no regular solution. 

One motivation for studying the regularity of solutions to the HMAE is the work of Semmes \cite{Semmes} and Donaldson \cite{Donaldson}, which shows that the geodesic equation in the space of K\"ahler metrics on a compact K\"ahler manifold $(X,\omega)$ cohomologous to $\omega$ can be cast as a Dirichlet problem for the HMAE on $X\times A$ where $A$ is an annulus. In this way the question of geodesic connectivity is translated into a question of regularity for solutions to the HMAE. In \cite{LempertVivas} Lempert-Vivas found such boundary data for which the solution failed to be regular, and thus showed that not all pairs of cohomologous K\"ahler metrics can be connected by a geodesic (see also later work by Darvas-Lempert \cite{DarvasLempert} and Darvas \cite{Darvas}).  

In \cite{DarvasLempert} examples are given for which the solution (again with the base being an annulus rather than a disc) does not have continuous second derivative.  We expect that the examples considered here in fact fail to be twice differentiable at any point $(z,1)$ where $z\in \mathbb P^1$ is a self-intersection point of the boundary of a simply connected Hele-Shaw domain $\Omega_t$.  Finally, we refer the reader to \cite[Chapter 2]{Guedjbook} for a discussion of the analogous problem of the complex HMAE for domains in $\mathbb C^n$.

The connection between the HMAE and the Hele-Shaw flow has been studied previously by the authors in \cite{RossNystromHeleShaw} (see also \cite{DonaldsonNahm} for another connection between the HMAE and free boundary problems).    The results there are in the opposite direction to those here, in that we use known regularity results of the HMAE (and thus the existence of the associated foliation by holomorphic discs) to prove short time regularity of the Hele-Shaw flow.     

\subsection*{Acknowledgements} 

We would particularly like to thank Bo Berndtsson for his close reading of the first version of this paper which has led to many improvements.  We also thank Robert Berman, Zbigniew B\l ocki, Laszlo Lempert and Yanir Rubinstein for discussions concerning this work.

\section{The Hele-Shaw flow} \label{Sec:HS}

\subsection{Definition and Basic Properties}

Assume $\phi$ is a smooth K\"ahler potential on $\mathbb P^1$, so $\omega_\phi:=\omega_{FS} + dd^c\phi$ is K\"ahler.  Given $t\in [0,1]$ we define 
\begin{equation}
\psi_t = \sup\{ \psi: \psi \text{ is usc and } \omega_{\psi}\ge 0 \text{ and } \psi\le \phi \text{ and } \nu_0(\psi)\ge t\}.\label{eq:envelope}
\end{equation}
Here $\omega_{\psi}:=\omega_{FS}+dd^c\psi$ and $\nu_0$ denotes the Lelong number, so $\nu_0(\psi)\ge t$ means that there is a constant $C$ such that $\psi(z) \le t\ln |z|^2 +C$ for all $z$ near $0\in \mathbb C\subset \mathbb P^1$. As the usc regularization of $\psi_t$ is itself a candidate for the envelope we see that $\psi_t$ is usc.

We now define the Hele-Shaw flow by
\begin{equation}
\Omega_t:=\Omega_t^{\phi}:=\{z\in \mathbb P^1: \psi_t(z)<\phi(z)\}\label{eq:heleshawdef}.
\end{equation}

\begin{proposition}(Basic Properties of Hele-Shaw flow) \label{prop:basicHS}
\begin{enumerate}
\item \label{item1} $\Omega_0=\emptyset$ and $0\in \Omega_t$ for $t>0.$ 
\item \label{item2} $\Omega_t$ is open, connected and $\partial \Omega_t$ has measure zero.
\item \label{item3} $\psi_t$ is $C^{1,1}$ on $\mathbb P^1\setminus\{0\}$. 
\item \label{item4} $$\omega_{\psi_t} = (1-\chi_{\Omega_t}) \omega_{\phi} +t\delta_0$$
in the sense of currents. Here $\chi_A$ denotes the characteristic function of a set $A$, and $\delta_0$ the Dirac delta.
\item \label{item5} $$\Area(\Omega_t):=\int_{\Omega_t}\omega_{\phi}=t.$$  
\end{enumerate}
\end{proposition}
\begin{proof}
This techniques used for this are standard (see e.g. \cite{Hedenmalm}) but for the convenience of the reader we sketch some details.

Clearly $\psi_0\equiv \phi$ and thus $\Omega_0=\emptyset$ so the Lemma is obviously true for $t=0.$ Thus let $t>0.$ That $0\in \Omega_t$ is obvious. Since $\psi_t$ is upper semicontinuous $\psi_t<\phi$ in a neighbourhood of the origin. It follows from standard potential theory that $\psi_t+\ln(1+|z|^2)$ is harmonic in any open set contained in $\Omega_t\setminus\{0\},$ so in particular $\psi_t$ is smooth on any punctured disc $D$ centered at the origin contained in $\Omega_t$.  Define
$$\psi' = \sup\left\{\begin{array}{l} \psi \text{ is usc on } \mathbb P^1\setminus D\text{ with }{\omega_{FS}}_{|\mathbb{P}^1\setminus D}+dd^c\psi\geq 0\text{ and } \psi\leq \phi_{|\mathbb{P}^1\setminus D}
\\ 
\text{and }  \psi\leq \psi_t \text{ on }\partial D
\end{array}\right\}.$$

Since ${\psi_t}_{|\mathbb{P}^1\setminus D}$ is a candidate for the envelope $\psi'$ we get that $\psi'\geq {\psi_t}_{|\mathbb{P}^1\setminus D}.$ On the other hand, if $\epsilon>0$ then extending $\max(\psi_t,\psi'-\epsilon)$ by $\psi_t$ on $D$ gives an $\omega_{FS}$-subharmonic function on $\mathbb{P}^1$ which is thus a candidate for the envelope defining $\psi_t$.  Hence $\psi'\leq {\psi_t}_{|\mathbb{P}^1\setminus D}+\epsilon$ and so $\psi'={\psi_t}_{|\mathbb{P}^1\setminus D}.$ 

If we let $w=1/z$ then $\psi'(w)+\ln(1+|w|^2)$ now solves a standard free boundary problem on the disc $\mathbb P^1\setminus D$ with obstacle given by $\phi(w)+\ln(1+|w|^2)$ and boundary condition given by $\psi_t(w)+\ln(1+|w|^2)$ restricted to $\partial D.$ So (\ref{item3}) follows from standard theory of free boundary problems (e.g. \cite[Thm 2.3]{CaffarelliK}) and (\ref{item4}) follows from (\ref{item3}) on $\mathbb P^1\setminus D$.  On the other hand the function $t\ln |z|^2 +C$ for some constant $C$ is a candidate for the envelope defining $\psi_t$, so $\nu_0(\psi_t) = t$, giving (\ref{item4}) on all of $\mathbb P^1$.   Moreover (\ref{item5}) in turn follows from (\ref{item4}) by Stokes theorem. 

That $\Omega_t$ is open of course also follows from $\psi_t$ being continuous.   The proof that $\Omega_t$ is connected is as in \cite[Prop 2.6]{Hedenmalm}.   Finally the fact that $\partial \Omega_t$ has zero measure again follows from standard theory of free boundary problems, see e.g \cite[p296]{Caffarelli}. In fact even more is proved in \cite{Caffarelli}, namely that each component of the boundary consists of a finite number of rectifiable Jordan curves.
\end{proof}

\begin{remark}
  Item (4) in particular implies that the sets $\Omega_t$ together with $\omega_{\phi}$ contain the same information as the functions $\psi_t$, and so we also think of this latter collection as the Hele-Shaw flow.    Items (3), (4) and (5) of the above Proposition also follow from more general work of Berman \cite[Sec 4]{Bermanample}.
\end{remark}

Clearly by definition if $t\le t'$ then $\Omega_t\subset \Omega_{t'}$ so this is an increasing flow of subsets of $\mathbb P^1$.    Observe also that if $\phi$ is rotation invariant then so is the flow $\Omega_t$.

An important fact is that the family of functions $\psi_t$ is concave in $t.$

\begin{proposition}
For any given $z$ we have that $\psi_t(z)$ is concave, decreasing and continuous in $t$ for $t\in [0,1].$
\end{proposition}

\begin{proof}
Let us define $\psi_t:=\phi$ for $t<0.$ It is then clear that $\psi_t$ is concave in $t$ since if $t=at_1+(1-a)t_2$ where $a\in [0,1]$ and $t_1,t_2\in (-\infty,1]$ then $$a\psi_{t_1}+(1-a)\psi_{t_2}\leq \psi_t$$ simply because the LHS has at least Lelong number $t$ at the origin while being bounded from above by $\phi.$ That it is decreasing is obvious. That $\psi_t$ decreases with $t$ then implies that $\lim_{t \to s-}\psi_t$ is $\omega_{FS}$-subharmonic and thus one sees that $$\lim_{t \to s-}\psi_t=\psi_s,$$ i.e. $\psi_t$ is left-continuous in $t$. This combined with concavity implies continuity.  
\end{proof}

\subsection{Multiply Connected Hele-Shaw Domains} \label{subsection:multcon}

As already mentioned, the sets $\Omega_t$ have a physical interpretation.  They describe the flow obtained by injecting a fluid at a point between two parallel plates between which there is a medium with permeability inversely proportional to $\Delta(\phi+\ln(1+|z|^2))$.   As such it is intuitively clear that there will be $\phi$ for which this flow becomes multiply connected.   In fact suppose we arrange so $\Delta (\phi+\ln (1+|z|^2)$ is very small on some Jordan curve going through the origin, while $\Delta (\phi+\ln (1+|z|^2)$ being relatively large in two regions separated by the curve.  The the flow will then  cover the curve before having the chance to engulf either of the regions, thereby giving rise to multiply connected Hele-Shaw domains. We now prove that this does indeed happen for suitable choices of $\phi$.

\begin{proposition} \label{prop:notsc}
There exists a smooth K\"ahler potential on $\mathbb P^1$ whose associated Hele-Shaw flow has the following property: there exist two times $0<t_1<t_2<1$ such that for all $t\in (t_1,t_2)$ the Hele-Shaw domain $\Omega_t$ is \emph{not} simply connected.
\end{proposition}

\begin{proof}
We will construct a $\phi$ with the following property: there exist two times $0<t_1<t_2<1,$ a Jordan curve $\gamma$ which passes through the origin and two points $p,q$ on opposite sides of $\gamma$ such that $\Omega_{t_1}$ contains $\gamma$ while neither $p$ or $q$ lie in $\Omega_{t_2}$. By monotonicity of the flow, it follows immediately that for any $t\in (t_1,t_2)$ the Hele-Shaw domain $\Omega_t$ is multiply connected.

So pick a Jordan curve $\gamma$ in $\mathbb{P}^1$ which passes through the origin, and let $U_1$ and $U_2$ denote the two connected components of the complement of $\gamma.$ Let $f$ be a smooth nonnegative function on $\mathbb{P}^1$ which is zero in a neighbourhood $U_{\gamma}$ of $\gamma$ and such that $$\int_{U_1}f\omega_{FS}=\int_{U_2}f\omega_{FS}=1/2.$$ Let $\phi_f$ be a smooth function such that $$\omega_{\phi_f}=f\omega_{FS}.$$ Fix $$0<t_0<1/4,$$ and set
$$\psi_f := \sup\{ \psi: \omega_{\psi}\ge 0, \psi\le \phi_f, \nu_0(\gamma)\ge t_0\}.$$
So $\omega_{\psi_f}  = \chi_{\{\psi_f=\phi_f\}} \omega_{\phi_f}+t_0\delta_0$ and we conclude that $\psi_f-\phi_f$ is harmonic in $U_{\gamma}\setminus \{0\}$ while having a singularity at the origin. Clearly $\psi_f-\phi_f\le 0$ on $U_{\gamma}$. Thus we conclude from the maximum principle that in fact  $\psi_f-\phi_f<0$ in $U_{\gamma}$.

Now let
$$\psi_{\epsilon} := \sup\{ \psi: \omega_{\psi} \ge -\epsilon \omega_{FS}, \psi\le \phi_f, \nu_0(\psi)\ge t_0\}.$$
Then $\psi_{\epsilon}-\phi_f$ decreases to $\psi_f-\phi_f$ as $\epsilon$ tends to zero (as for any decreasing sequence of subharmonic functions the limit is also subharmonic, and thus a candidate for $\psi_f-\phi_f$).  In particular if $z$ is a fixed point on $\gamma$ then for $\epsilon$ sufficiently small $\psi_{\epsilon} - \phi_f$ is strictly negative at $z$, and hence by upper semicontinuity of $\psi_{\epsilon} - \phi_f$ this is true in a neighbourhood of $z$. Thus for $\epsilon\ll 1$ sufficiently small, $\psi_{\epsilon} - \phi_f$ is negative on the compact set $\gamma$.

Pick such a small $\epsilon$ and let
$$\phi := \frac{\phi_f}{1+\epsilon}.$$
Then $\omega_{\phi} = (1+\epsilon)^{-1} ( \epsilon \omega_{FS} + \omega_{\phi_f})$ so $\phi$ is a K\"ahler potential.  By construction it now follows that if we let $t_1:=(1+\epsilon)^{-1}t_0$ then $\gamma \subset \Omega^{\phi}_{t_1}$. Observe that $0<t_1<1/4.$

We have that
$$\Area_{\phi}(U_1):=\int_{U_1} \omega_{\phi} = \int_{U_1} \omega_{FS} + \frac{dd^c\phi_f}{1+\epsilon} \ge \frac{1}{1+\epsilon}\int_{U_1} \omega_{\phi_f} > 1/4,$$
and similarly $\Area_{\phi}(U_2)>1/4.$ Now if we pick some $t_2$ such that $0<t_1<t_2<1/4,$ then 
$$\Area_{\phi}(\Omega^{\phi}_{t_2}) =t_2 < 1/4 < \Area_{\phi}(U_1)$$ and similarly $$\Area_{\phi}(\Omega^{\phi}_{t_2})<\Area_{\phi}(U_2).$$ In particular the sets  $U_1\setminus \Omega^{\phi}_{t_2}$ and $U_2\setminus \Omega^{\phi}_{t_2}$ must both be nonempty, which allows us to pick points $p$ and $q$ in the complement of $\Omega^{\phi}_{t_2}$ on either side of $\gamma.$ This then concludes the proof. 
\end{proof}

\section{The Legendre Transform between the HMAE and the Hele-Shaw Flow}

We shall focus on a simple case of the complex Homogeneous Monge Amp\`ere equation.   Suppose that $\phi(\cdot,\tau)$ is a smooth family of K\"ahler potentials parameterized by $\tau\in \partial \mathbb D$.  We denote by $\pi$ the projection $\mathbb P^1\times \overline{\mathbb D}\to \mathbb P^1$.

\begin{definition}\label{def:weakhmae}
The \emph{solution} $\Phi$ to the Homogeneous Monge-Amp\`ere equation with boundary data $\phi(\cdot,\tau)$ is the function on $\mathbb P^1\times \overline{\mathbb D}$ given by
$$ \Phi = \sup \{ \psi : \psi \text{ is usc and }\pi^*\omega_{FS} + dd^c\psi \ge 0\text{ and } \psi(\cdot,\tau)\le \phi(\cdot,\tau) \text{ for } \tau\in \partial\mathbb D\}.$$
\end{definition}

From general theory $\pi^*\omega_{FS}  + dd^c\Phi\ge 0$ as a current, has the boundary value $\Phi(\cdot,\tau) = \phi(\cdot,\tau)$ for $\tau\in\partial \mathbb D$ and solves the equation
$$(\pi^*\omega_{FS}+ dd^c\Phi)^{2} =0$$ 
in the sense of Bedford-Taylor.  As is well known, due to Chen \cite{Chen} with complements by B\l ocki \cite{Blocki}, the solution is almost $C^{1,1},$ and in fact since in our case $(\mathbb{P}^1,\omega_{FS})$ has nonnegative sectional curvature the solution is truly $C^{1,1}$ by the result of B\l ocki \cite[Thm. 1.4]{Blocki} (see also the recent work of Berman \cite{Bermanreg} for a proof the weak solution is $C^{1,1}_{loc}$ along the original lines of Bedford-Taylor). \medskip

Recall $\rho$ denotes the usual $\mathbb{C}^{\times}$-action on $\mathbb{P}^1$ which acts by multiplication on $\mathbb{C}\subset \mathbb{P}^1.$   Letting $\phi$ be a K\"ahler potential as before,  we wish to consider the function 
$$\phi(z,\tau):=\phi(\rho(\tau)z)$$
as boundary data to the HMAE on $\mathbb{P}^1\times \overline{\mathbb{D}}.$  Let $\Phi$ denote the weak solution to the corresponding HMAE as in \eqref{def:weakhmae}.  The goal in this section is to show that $\Phi$ is connected via a Legendre transform to the Hele-Shaw flow on $\mathbb{P}^1$ taken with respect to $\phi$.  \medskip

To do this, consider the envelope on $\mathbb P^1\times \overline{\mathbb D}$ given by
$$ \tilde{\Phi} = \sup \{ \psi : \psi\text{ is usc, } \pi^*\omega_{FS} + dd^c\psi\ge 0 \text{, } \psi(z,\tau)\le \phi(z) \text{ for } \tau\in \partial \mathbb D \text{ and } \nu_{(0,0)}(\psi)\ge 1\}$$
(so the boundary data is independent of $\tau$, and the $\psi$ have Lelong number at least one at the point $(0,0)$).  Then by standard arguments,  $\tilde{\Phi}$ is usc,  $\pi^*\omega_{FS} + dd^c\tilde{\Phi}\ge 0$ and $(\pi^*\omega_{FS} + dd^c\tilde{\Phi})^2=0$ away from $(0,0)$.  

\begin{remark}
  The function $\tilde{\Phi}$ solves the HMAE over the punctured disc $\overline{\mathbb D}^{\times}$ with boundary data independent of $\tau$, and thus is a weak geodesic ray emanating from $\phi$.
\end{remark}

\begin{proposition} \label{propequiv}
We have that $$\Phi(z,\tau)+\ln|\tau|^2+\ln(1+|z|^2)=\tilde{\Phi}(\tau z,\tau)+\ln(1+|\tau z|^2) \text{ for } (z,\tau)\in \mathbb P^1\times \overline{\mathbb D}^{\times}.$$
\end{proposition}

\begin{proof}
Consider the space $\mathcal{X}$ we get by blowing up $\mathbb{P}^1\times \overline{\mathbb{D}}$ at the point $(\infty,0)$. Let $\mu_1$ denote the modification map, and let $E_1$ denote the exceptional divisor. The central fibre thus consists of two copies of $\mathbb{P}^1$. We call the other one $E_2,$ and if we blow down this one we get again $\mathbb{P}^1\times \overline{\mathbb{D}}$. Call this modification map $\mu_2,$ so $E_2$ is now the exceptional divisor of $\mu_2.$ One can now check that 
$$\mu_2 \circ \mu_1^{-1}(z,\tau)=(\tau z,\tau),$$ except for being undefined at the point $(\infty,0).$

We have two ways to pull back the Fubini-Study form to $\mathcal{X}$, via $\mu_1$ or $\mu_2$, and we want to describe how they are related. For this let $f$ denote the function on $\mathcal{X}$ defined as $$f:=(\ln(|\tau|^2+|z|^2)-\ln(1+|z|^2))\circ \mu_2.$$ Then a simple calculation shows that $$dd^c f=\mu_1^*(\pi^*\omega_{FS})-(\mu_2^*(\pi^*\omega_{FS})-[E_2]),$$ where $[E_2]$ denotes the current of integration along $E_2.$ This then shows that $$\Phi\circ \mu_1 \circ \mu_2^{-1}+\ln(|\tau|^2+|z|^2)-\ln(1+|z|^2)$$ is a candidate for the envelope defining $\tilde{\Phi},$ while reversely $$\tilde{\Phi}\circ \mu_2 \circ \mu_1^{-1}+f \circ \mu_1^{-1}$$ is a candidate for the envelope defining $\Phi.$ It follows that $$\Phi\circ \mu_1+f=\tilde{\Phi}\circ \mu_2$$ and we simply get the formula of the Proposition by first taking $\mu_1^{-1}.$
\end{proof}

\begin{remark}
  From the above we see that
$$ \Phi(z,\tau) - \tilde{\Phi}(\tau z,\tau) = \ln \left(\frac{1 + |\tau z|^2}{|\tau|^2(1+|z|^2)}\right)$$
which is smooth on $\mathbb P^1\times \overline{\mathbb D}^{\times}$.  Hence as $\Phi$ is $C^{1,1}$ on $\mathbb P^1\times \overline{\mathbb{D}}^{\times}$, the same is true of $\tilde{\Phi}$.  
\end{remark}

\begin{lemma}
\begin{equation} \label{phiinequality}
\tilde{\Phi}(z,\tau)\leq \ln(|\tau|^2+|z|^2)- \ln(1+|z|^2) + \max(\phi).
\end{equation}
\end{lemma}

\begin{proof}
First assume that $\phi$ is identically zero. Then clearly $\Phi$ is also identically zero so from Proposition \ref{propequiv} we see that $$\tilde{\Phi}(z,\tau)=\ln(|\tau|^2+|z|^2)- \ln(1+|z|^2).$$ The lemma then follows from the obvious monotonicity property of $\tilde{\Phi}.$
\end{proof}

Now we wish to connect $\tilde{\Phi}$ (and hence $\Phi$) with the Hele-Shaw flow of $\phi$.  For any $t\in \mathbb{R}$ we let 
\begin{equation}
\psi_t = \sup\{ \psi: \psi \text{ is usc and } \omega_{\psi}\ge 0 \text{ and } \psi\le \phi \text{ and } \nu_0(\psi)\ge t\}.
\end{equation}
Note that for $t<0,$ $\psi_t=\phi$ while if $t>1$ we get that $\psi_t\equiv -\infty.$ For $t\in [0,1]$ we recognize $\psi_t$ as the envelopes that contain the same data as the Hele-Shaw flow $\Omega_t = \{ \psi_t<\phi\}$. 

We saw in Section \ref{Sec:HS} that for a fixed $z$ the function $\psi_t(z)$ is concave in $t.$ On the other hand, the function $\tilde{\Phi}(z,e^{-s/2})$ is subharmonic and independent of the imaginary part of $s$ and thus $\tilde{\Phi}(z,e^{-s/2})$ is convex in $s$ (where we now think of $s$ as a real variable taking values in $[0,\infty)$). We can then define $\tilde{\Phi}(z,e^{-s/2})$ to be $+\infty$ for $s<0$ to get a convex function defined on the whole of $\mathbb{R}.$ 

Now we recall the definition of the (one-variable) Legendre transform (or convex conjugate).

\begin{definition}
Given a function $u: \mathbb{R} \to \mathbb{R}\cup \{+\infty\}$ the Legendre transform $\hat{u}: \mathbb{R} \to \mathbb{R}\cup \{+\infty\}$ is the convex function defined as $$\hat{u}(y):=\sup_{x\in \mathbb{R}^n}\{xy-u(x)\}.$$
\end{definition}

The Fenchel-Moreau Theorem (see e.g. \cite{Rockafellar}) now asserts that the Legendre transform is an involution (i.e. $\hat{\hat{u}}=u$) precisely on the set of convex lower semicontinuous functions.

The next Theorem says that $u(s):=\tilde{\Phi}(z,e^{-s/2})+s$ is the Legendre transform of $-\psi_t(z)$ and vice versa.

\begin{theorem}
\begin{equation} \label{legendre1}
\psi_t(z)=\inf_{|\tau|>0}\{\tilde{\Phi}(z,\tau)-(1-t)\ln|\tau|^2\}
\end{equation} 
and 
\begin{equation} \label{legendre2}
\tilde{\Phi}(z,\tau)=\sup_{t}\{\psi_{t}(z)+(1-t)\ln |\tau|^2\}.
\end{equation}
\end{theorem}

\begin{proof}
 First we note that $$\psi_{t}(z)+(1-t)\ln |\tau|^2$$ is a candidate for the envelope defining $\tilde{\Phi}$ so by definition we get that the LHS is less than or equal to the RHS in (\ref{legendre1}). Since $\tilde{\Phi}(z,\tau)$ is independent of the argument of $\tau$ it follows from Kiselman's minimum principle (see \cite{Kiselman}) that the RHS defines an $\omega_{FS}$-sh function on $\mathbb{P}^1,$ which we will denote by $\tilde{\psi}_{t}.$ Clearly $\tilde{\psi}_{t}\leq \phi$, so if we can show that it has logarithmic singularity of order at least $t$ at the origin, then it would follow from the definition of $\psi_{t}$ as the supremum of all such functions that $\tilde{\psi}_{t}\leq \psi_{t}$. Since for a fixed $z$  the function $\tilde{\Phi}(z,e^{-s/2})$ is subharmonic and independent of the imaginary part of $s$ we get that $\tilde{\Phi}(z,e^{-s/2})$ is convex in $s$ (where we now think of $s$ as a real variable taking values in $[0,\infty)$). Thus for a fixed $z$ $$-\tilde{\psi}_t(z)=\inf_{s\geq 0}\{\tilde{\Phi}(z,e^{-s/2})+(1-t)s\}=\sup_{s \in \mathbb{R}}\{ts-(\tilde{\Phi}(z,e^{-s/2})+s)\}$$ is the Legendre transform of a convex function. Using the inequality (\ref{phiinequality}) we get that $$\tilde{\psi}_{t}(z)\leq \inf_{s\geq 0}\{\ln (e^{-s}+|z|^2)+(1-t)s\}-\ln(1+|z|^2)+\max(\phi).$$ 
By elementary means one easily checks that if $t\in (0,1)$ then
$$\inf_{s\geq 0}\{\ln(e^{-s} + |z|^2)+(1-t)s\}=t(\ln|z|^2-\ln t)-(1-t)\ln (1-t)$$ which shows that indeed $\tilde{\psi}_{t}$ has a logarithmic singularity of order $t,$ at least when $t\in (0,1).$  For $t=0$ one notes that $\phi(z)+\ln |\tau|^2$ is a candidate for the envelope defining $\tilde{\Phi}$, so $\phi(z) + \ln |\tau|^2\le \tilde{\Phi}$, which implies that $\tilde{\psi}_{0}=\phi$.  For $t=1$, the fact that $$\tilde{\psi}_1(z)\leq \tilde{\Phi}(z,0)$$ implies that it has the right singularity as well. The case that $t<0$ and $t>1$ are immediate, so this proves (\ref{legendre1}).

We thus see that $$-\psi_t(z)=\sup_{s \in \mathbb{R}}\{ts-(\tilde{\Phi}(z,e^{-s/2})+s)\},$$ i.e. that $-\psi_t(z)$ is the Legendre transform of $u(s):=\tilde{\Phi}(z,e^{-s/2})+s.$ We also know that $u(s)$ is convex and lower semicontinuous (since it is continuous on $[0,\infty)$ and constantly $-\infty$ on $(-\infty,0)$) so by the Fenchel-Moreau Theorem we get that $u(s)$ is the Legendre transform of $-\psi_t(z).$ This is exactly what is asserted in (\ref{legendre2}). 
\end{proof}

\subsection{The Hamiltonian}

We will have use for the function $H$ on $\mathbb{P}^1\times \overline{\mathbb{D}}^{\times}$ defined as $$H(z,\tau):=\frac{\partial}{\partial s} \tilde{\Phi}(z,e^{-s/2}),$$ where $s:=-\ln |\tau|^2$ (when $|\tau|=1$ and thus $s=0$ we take the right derivative). As $\tilde{\Phi}$ is $C^{1,1}$ on $\mathbb{P}^1\times \overline{\mathbb{D}}^{\times}$ the function $H$ is well-defined and continuous (even Lipschitz but we will not need this). 

\begin{proposition}\label{prop:Hu_t}
$$H(z,1)+1=\sup\{t: \psi_{t}(z)=\phi(z)\} = \sup\{ t: z\notin \Omega_t\}.$$
\end{proposition}

\begin{proof}
From (\ref{legendre2}) we see that if $\psi_{t}(z)=\phi(z)$ then $$\tilde{\Phi}(z,e^{-s/2})\geq (t-1)s+\phi(z)$$ and thus $$H(z,1)\geq\sup\{t: \psi_{t}(z)=\phi(z)\}-1.$$ Recall that for a fixed $z$ the function $\psi_{t}(z)$ is concave and decreasing in $t.$ Thus if $\psi_{t}(z)\leq a<0$ then it follows from (\ref{legendre2}) that $$\tilde{\Phi}(z,e^{-s/2})\leq \max( (t-1)s,-a)+\phi(z)$$ and so $H(z,1)\leq t-1,$ which proves the proposition.
\end{proof}

\begin{proposition} \label{prop:ham2}
For $0<|\tau|<1$ we have that $$H(z,\tau)=t-1 \qquad \Longleftrightarrow \qquad \tilde{\Phi}(z,\tau)=\psi_t(z)-(1-t)\ln |\tau|^2.$$
\end{proposition}

\begin{proof}
Fix a point $(z_0,\tau_0),$  $0<|\tau_0|<1$. From (\ref{legendre2}) we see that $$\tilde{\Phi}(z_0,\tau_0)=\sup_{t\in [0,1]}\{\psi_t(z_0)+(1-t)\ln |\tau_0|^2\}.$$ Since $\psi_t(z_0)$ is continuous in $t$ we must have that $\tilde{\Phi}(z_0,\tau_0)=\psi_{t_0}(z_0)+(1-t_0)\ln |\tau_0|^2$ for some $t_0\in [0,1].$ Since we always have that $$\tilde{\Phi}(z,e^{-s/2})\geq \psi_{t_0}(z)-(1-t_0)s$$ it follows that $$H(z_0,\tau_0)=\frac{\partial}{\partial s}_{|s=-\ln|\tau_0|^2}(\psi_{t_0}(z)-(1-t_0)s)=t_0-1.$$
\end{proof}

\section{Harmonic discs}

As above, let $\Phi$ be the weak solution to the HMAE with boundary data $\phi(z,\tau) = \phi(\rho(\tau)z)$. Recall that if $g\colon \mathbb D\to \mathbb P^1$ is holomorphic then we say the graph of $g$ is a harmonic disc if $\Phi$ is $\pi^*\omega_{FS}$-harmonic along this graph.

\begin{theorem} \label{thm:holdiscs}
Let $g\colon \mathbb D\to \mathbb P^1$ be holomorphic.  Then the graph of $g$ is a harmonic disc of $\Phi$ if and only if either (1) $g\equiv 0$ or (2) $g(\tau)=\tau^{-1}z$ where $z\in \Omega_1^c$ or (3) $\tau \mapsto \tau g(\tau)$ is a Riemann mapping for a simply connected Hele-Shaw domain $\Omega_t$ that maps $0\in \mathbb D$ to $0\in \Omega_t$.  The function $H$ is constant along the associated discs $\{(\tau g(\tau),\tau)\}$, in the first case $H= -1,$ in the second case $H= 0$ while in the third case $H= t-1.$
\end{theorem}

\begin{lemma} \label{lem:proper}
If $f\colon D_1\to D_2$ is a proper holomorphic map between two open domains in $\mathbb{P}^1$ then the number of preimages $N_p:=\#\{f^{-1}(p)\}$ (counted with multiplicity) is constant. 
\end{lemma}

\begin{proof}
Let $\gamma$ be a smooth curve in $D_2$ connecting two points $p$ and $q$ and let $U$ be a finite union of open discs compactly supported in $D_1$ which together cover the compact set $f^{-1}(\gamma)$. Since the image of any boundary component of $U$ cannot cross $\gamma$  the winding numbers of the image of any such boundary component with respect to $p$ and $q$ must be the same.  Since that winding number counts the number of preimages inside that component we get by adding up the winding numbers for the different boundary components that $N_p=N_q$.
\end{proof}

\begin{proof}[Proof of Theorem \ref{thm:holdiscs}]
 
First assume that $\Phi$ is $\omega_{FS}$-harmonic along a holomorphic disc $\{(g(\tau),\tau)\}.$ We want to show that $g(\tau)$ must be of the form stated in the theorem. 

Let $f(\tau):=\tau g(\tau).$ From Proposition \ref{propequiv} we have $\tilde{\Phi}$ is $\omega_{FS}$-harmonic along the punctured disc $\{(f(\tau),\tau): 0<|\tau|<1\}$ and since $\tilde{\Phi}$ is bounded away from $(0,0)$ it will be $\omega_{FS}$-harmonic along the unpunctured disc $\{(f(\tau),\tau): |\tau|<1\}$ unless that disc passes through $(0,0)$. If it does pass through $(0,0)$, the restriction of $\tilde{\Phi}$ to this disc has a logarithmic singularity of order at most one at $(0,0)$,  because $\tilde{\Phi}(z,\tau)$ is bounded from below by $\phi(z)+\ln|\tau|^2.$

Pick a $\tau_0,$ $0<|\tau_0|<1$ and let $z_0:=f(\tau_0)$ and $t_0:=H(z_0,\tau_0).$  The restriction of $$\psi_{t_0}(z)+(1-t_0)\ln |\tau|^2-\tilde{\Phi}(z,\tau)$$ to the disc is clearly subharmonic and less than or equal to zero and from Proposition \ref{prop:ham2} we see that it is equal to zero at $(z_0,\tau_0)$. By the maximum principle we thus get that 
\begin{equation} \label{eq:disceq}
\tilde{\Phi}(f(\tau),\tau)=\psi_{t_0}(f(\tau))+(1-t_0)\ln |\tau|^2
\end{equation} 
for all $\tau.$ 

We now consider the case when $f$ is constant. Putting $f(\tau)=z_0$ in (\ref{eq:disceq}) we get that 
\begin{equation} \label{eq:disceq2}
\tilde{\Phi}(z_0,\tau)=\psi_{t_0}(z_0)+(1-t_0)\ln |\tau|^2.
\end{equation} 
One possible case is that $z_0=0,$ and in fact we will see later that $\tilde{\Phi}$ always is $\omega_{FS}$-harmonic along this disc. If $z_0\neq 0,$ then we know that the LHS is $\omega_{FS}$-harmonic along the whole unpunctured disc, which forces $t_0=1$ (otherwise the RHS would be singular at the center). Letting $\tau \to 1$ in (\ref{eq:disceq2}) we get that $$\psi_1(z_0)=\lim_{\tau \to 1}\tilde{\Phi}(z_0,\tau)=\phi(z_0)$$ by continuity of $\tilde{\Phi}$ which shows that $z_0\in \Omega_1^c.$ 

We move on to the case when $f$ is nonconstant. By (\ref{eq:disceq}) we then get that $\psi_{t_0}$ is $\omega_{FS}$-harmonic in a neighbourhood of any point $f(\tau),$ $|\tau|>0,$ which implies that $f(\tau)\in \Omega_{t_0}.$ Note that this rules out the possibility of $t_0=0,$ which in turn implies that $f(\tau)\neq 0$ whenever $\tau \neq 0$ as this would make the RHS in (\ref{eq:disceq}) be $-\infty$. If $f(0)\neq 0$ then $\tilde{\Phi}$ is $\omega_{FS}$-harmonic along the entire disc and the same argument as above shows that $f(0)\in \Omega_{t_0}$ while $f(0)=0$ of course also implies that $f(0)\in \Omega_{t_0}$. We thus have that $f$ maps $\mathbb{D}$ to $\Omega_t.$ 

We now claim that $f$ is proper. Because if we choose a sequence $\tau_i$ such that $|\tau_i|\to 1$ then by continuity of $\tilde{\Phi}$ and (\ref{eq:disceq}) we get that 
\begin{equation} \label{eq:properness}
\lim_{i\to \infty} (\psi_{t_0}(f(\tau_i))-\phi(f(\tau_i))=0.
\end{equation} 
Since $\Omega_{t_0}$ is exhausted by the compact sets $\{z: \psi_{t_0}(z)\leq \phi(z)-1/n\}$ by (\ref{eq:properness}) $f(\tau_i)$ escapes to infinity in $\Omega_{t_0}.$ This shows that $f$ is proper. 

We now want to calculate $N_0:=\#\{f^{-1}(0)\}$. Because $f(\tau)\neq 0$ whenever $\tau \neq 0$,  $N_0$ equals multiplicity of the zero at zero. If $f(0)\neq 0$ we have $N_0=0$, which by Lemma \ref{lem:proper} is impossible, so we conclude that $f(0)=0.$ From (\ref{eq:disceq}) we get that $$\psi_{t_0}(f(\tau))+(1-t_0)\ln |\tau|^2=\tilde{\Phi}(f(\tau),\tau)\geq \phi(f(\tau))+\ln |\tau|^2.$$ If $m$ is the multiplicity of the zero of $f$ at zero we get that the LHS has Lelong number $1+(m-1)t_0$ while the RHS has Lelong number one, which shows that $N_0=m=1.$ This via Lemma \ref{lem:proper} implies that $f$ is a bijection between $\mathbb{D}$ and $\Omega_{t_0}.$

Since $\tau g(\tau)=f(\tau)$ this concludes the proof of the first direction of the Theorem.

We now prove that if $g(\tau)$ is of the form specified in the Theorem then $\Phi$ is $\omega_{FS}$-harmonic along the graph of $g.$ As before we let $f(\tau):=\tau g(\tau)$ and it is clearly enough to show that $\tilde{\Phi}$ is $\omega_{FS}$-harmonic along the punctured disc $\{(f(\tau),\tau): 0<|\tau|<1\}.$

The first case was $g\equiv 0$ and thus $f\equiv 0.$ From (\ref{legendre2}) we see that $\tilde{\Phi}(0,\tau)=\phi(0)+\ln |\tau|^2$ and thus $\tilde{\Phi}$ is indeed $\omega_{FS}$-harmonic along $\{(0,\tau): 0<|\tau|<1\}$. The second case corresponded to $f(\tau)=z$ where $z\in \Omega_1^c.$ From (\ref{legendre2}) we get that $\tilde{\Phi}(z,\tau)\geq \psi_1(z)=\phi(z).$ On the other hand the opposite inequality always holds and hence $\tilde{\Phi}$ is constant along the disc $\{(z,\tau): |\tau|<1\}.$ The third case meant that $f$ was a bijection between $\mathbb{D}$ and some simply connected Hele-Shaw domain $\Omega_t$, and $f(0)=0$. From (\ref{legendre2}) we get that $$\tilde{\Phi}(f(\tau),\tau)\geq \psi_t(f(\tau))+(1-t)\ln|\tau|^2.$$ The LHS is $\omega_{FS}$-subharmonic with Lelong number one at $\tau=0$ while the RHS is $\omega_{FS}$-harmonic except at $\tau=0$ where it also has Lelong number one. Since $f(\tau)$ escapes to infinity in $\Omega_t$ as $|\tau| \to 1$ the RHS approaches the LHS as $\tau$ approaches the boundary. By the maximum principle we get the equality $$\tilde{\Phi}(f(\tau),\tau)= \psi_t(f(\tau))+(1-t)\ln|\tau|^2,$$ which shows that $\tilde{\Phi}$ is $\omega_{FS}$-harmonic along the punctured disc $\{(f(\tau),\tau): 0<|\tau|<1\}.$ This concludes the proof.
 
\end{proof}

\section{Restatement of Main Theorems}\label{sec:example}

We saw in Section \ref{subsection:multcon} that for some $\phi$ the Hele-Shaw domains $\Omega_t$ are multiply connected for all $t$ in some interval. Combining this with our main theorem yields the following. 

\begin{theorem} \label{thm:partialsmooth}
Let $\phi$ be a K\"ahler potential on $\mathbb P^1$ whose Hele-Shaw flow satisfies the conclusion of Proposition \ref{prop:notsc}, i.e. there are two times $0<t_1<t_2<1$ such that for any $t\in (t_1,t_2)$ $\Omega_t$ is multiply connected.  Let $\Phi$ be the solution to the HMAE on $\mathbb{P}^1\times \overline{\mathbb{D}}$ with $\phi(z,\tau):=\phi(\rho(\tau)z)$ as boundary data. Then for such $\Phi$ there exists an open set $U$ in $\mathbb{P}^1\times \overline{\mathbb{D}}$ with nonempty intersection with $\mathbb{P}^1\times \partial \mathbb{D}$ which does not intersect any harmonic disc of $\Phi$.
\end{theorem}

\begin{proof}
By Theorem \ref{thm:holdiscs} no harmonic disc of $\Phi$ can intersect the open set $$U:=\{(z,\tau): t_1-1<H(\tau z,\tau)<t_2-1, |\tau|>0\}.$$ Since $H(z,1)$ clearly attains both values $-1$ and $0$ it follows from continuity that $U\cap(\mathbb{P}^1\times \partial \mathbb{D})$ is nonempty.  
\end{proof}

\begin{theorem}\label{thm:perturb}
Let $\phi$ be as in Theorem \ref{thm:partialsmooth}. Then there exist a nonempty open set $U'$ in $\mathbb{P}^1\times \mathbb{D}$ and an $\epsilon>0$ such that if $\phi'(z,\tau)$ is any smooth boundary data with $$\|\phi'-\phi\|_{C^2(\mathbb{P}^1\times \partial \mathbb{D})}<\epsilon$$ and $\Phi'$ is the associated solution to the HMAE, then no harmonic disc of $\Phi'$ can pass through $U'.$
\end{theorem}
   
\section{Proof of Theorem \ref{thm:perturb}}

Recall from the Introduction that a solution $\Phi$ to the HMAE is called a \emph{regular solution} if it is smooth and $\omega + dd^c\Phi(\cdot,\tau)$ is a K\"ahler form for all $\tau\in \overline{\mathbb D}$. To prove Theorem \ref{thm:perturb} we will need Donaldson's Openness Theorem:

\begin{theorem}[Donaldson {\cite[Theorem 1]{Donaldson}}]
  The set of smooth boundary data $\phi: X\times \partial \mathbb{D} \to \mathbb{R}$ for which the HMAE has a regular solution is open in the $C^2$ topology. 
\end{theorem}

\begin{proof}[Proof of Theorem \ref{thm:perturb}]
Let $U$ and $\Phi$ be as above, and $U'$ be an nonempty open set which is relatively compact in $U\cap ((\mathbb{P}^1\setminus \{0\})\times \mathbb{D}^{\times})$. We argue by contradiction, so assume that there is a sequence of solutions $\Phi_k$ with boundary values $\phi_k$ such that $$\|\phi_k-\phi(z,\tau)\|_{C^2(\mathbb{P}^1\times \partial \mathbb{D})}<1/k$$ and such that for each $k$ the graph of some holomorphic $g_k\colon \mathbb D\to \mathbb P^1$ is a harmonic disc of $\Phi_k$ which pass through $U'.$  

Our first claim is that there exists an $r>0$ such that for $k$ large $$\inf\{|g_k(\tau)|: \tau \in \mathbb{D}\}\geq r.$$ For simplicity we will assume that $\omega_{\phi}=\delta\omega_{FS}$ in $D_{r'}:=r'\mathbb{D}$ for some $\delta,r'>0,$ as was the case for the $\phi$ used to prove Proposition \ref{prop:notsc}. Then the function $h(z):=\phi(z)+(1-\delta)\ln(1+|z|^2)$ is harmonic in $D_{r'}.$   By symmetry, it is not hard to see that the Hele-Shaw domains (defined by our potential $\phi$) are initially just concentric discs centered at the origin, at least up to the point where $\Omega_t=D_{r'}.$ From Theorem \ref{thm:holdiscs} we see that this implies that for any fixed $z\in D_{r'},$ the set $\{(z,\tau): \tau\in \mathbb{D}\}$ is a harmonic disc associated to $\Phi$.  Thus $\Phi(z,\tau)$ is harmonic in $\tau$ for these fixed $z$.

 This then implies that $$\Phi(z,\tau)=-(1-\delta)\ln(1+|z|^2)+h(\tau z)\text{ on }  D_{r'}\times \overline{\mathbb{D}}.$$

Pick a smooth K\"ahler potential $u$ on $\mathbb{P}^1$ which is rotation invariant, and such that on $D_{r'}$ we have that $$u\leq -(1-\delta)\ln(1+|z|^2)$$ with equality on $D_{r'/2}$ but strict inequality on $\partial D_{r'}.$ Clearly $\Psi(z,\tau):=u(z)$ is a regular solution to the HMAE and we note that 
\begin{align}\label{eq:perturb1}
\Psi(z,\tau)+h(\tau z)=\Phi(z,\tau)&\text{ on } D_{r'/2}\times \overline{\mathbb{D}}, \\
\Psi(z,\tau)+h(\tau z)\leq \Phi(z,\tau)& \text{ on } D_{r'}\times \overline{\mathbb{D}},\nonumber \\
\Psi(z,\tau)+h(\tau z)<\Phi(z,\tau)& \text{ on }\partial D_{r'}\times \overline{\mathbb{D}}. \nonumber
\end{align}

\noindent (Even if one loses the assumption that $\omega_{\phi} = \delta\omega_{FS}$ near $0$ then arguing as in \cite{RossNystromHeleShaw} one can still find an $r'>0$ and harmonic $h$ and regular solution $\Psi$ such that the three statements of \eqref{eq:perturb1} hold).

Let $f_k(z,\tau):=\phi_k(z,\tau)-\phi(z,\tau)$  and let $\Psi_k$ denote the solution to the HMAE with boundary data $\psi_k(z,\tau):=u(z)+f_k(z,\tau).$ Since by assumption $$\|f_k\|_{C^2(\mathbb{P}^1\times \partial \mathbb{D})}<1/k$$ it follows from Donaldson's Openness Theorem that for large $k$ the function $\Psi_k$ is regular, and that the harmonic discs of $\Psi_k$ converge to those of $\Psi.$ In particular there exists an $r>0$ such that for any large enough $k$ there is a neighbourhood $U_k$ of $D_r\times \mathbb{D}$ that is relatively compact in $D_{r'/2}\times \mathbb{D}$ which is foliated by harmonic discs of $\Psi_k.$ We also have that for large $k$ 
\begin{align*}
\Psi_k(z,\tau)+h(\tau z)=\Phi_k(z,\tau)&\text{ on }  D_{r'/2}\times \partial\mathbb{D}, \\
\Psi_k(z,\tau)+h(\tau z)\leq \Phi_k(z,\tau)&\text{ on }  D_{r'}\times \partial\mathbb{D}, \\
\Psi_k(z,\tau)+h(\tau z)<\Phi_k(z,\tau)& \text{ on }\partial D_{r}\times \overline{\mathbb{D}}, 
\end{align*} 
and therefore $$\Psi_k(z,\tau)+h(\tau z)\leq \Phi_k(z,\tau)\text{ on } D_{r'}\times \overline{\mathbb{D}}.$$ On $U_k$ though we must have $$\Psi_k(z,\tau)+h(\tau z)\geq  \Phi_k(z,\tau)$$ since along any harmonic disc of $\Psi_k$ with boundary in $D_{r'/2}\times \partial \mathbb{D}$ the LHS is $\pi^*\omega_{FS}$-harmonic, the RHS is $\pi^*\omega_{FS}$-subharmonic, and they agree on the boundary. In conclusion $$\Psi_k(z,\tau)+h(\tau z)=\Phi_k(z,\tau)$$ on $U_k$ which shows $\Phi_k$ is regular on $U_k$ and that $U_k$ is foliated by harmonic discs of $\Phi_k.$ 

By shrinking $U'$ if necessary we may assume $U'\cap (D_{r'/2}\times \mathbb D)  = \emptyset$, and so
 $$U'\cap U_k=\emptyset \text{ for all }k.$$ Since the graph of $g_k$ by assumption passes through $U'$ it is not one of the harmonic discs that foliate $U_k$ and since $\Phi_k$ is regular this means that the graph of $g_k$ cannot intersect $U_k.$ Finally this implies that $|g_k(\tau)|\geq r$ for all $k$ as claimed.\medskip 

Thus $1/g_k$ is a bounded family of holomorphic functions on $\mathbb{D}$ so by Montel's theorem it is normal and thus replacing $g_k$ by a subsequence we may assume $g_k$ converges uniformly on compacts to a holomorphic function $g.$ We claim that the graph of $g$ then is a harmonic disc of $\Phi.$

To see this let $$h_k(\tau):=\Phi_k(g_k(\tau),\tau)+\ln(1+|1/g_k(\tau)|^2)$$ and $$h(\tau):=\Phi(g(\tau),\tau)+\ln(1+|1/g(\tau)|^2).$$ By assumption $h_k$ is harmonic and we want to show that $h$ is harmonic (which says precisely that the graph of $g$ is a harmonic disc of $\Phi.$).  Now 
\begin{eqnarray} \label{eq:harmdiff}
|h_k(\tau)-h(\tau)|\leq |\Phi_k(g_k(\tau),\tau)-\Phi(g_k(\tau),\tau)|+|\Phi(g_k(\tau),\tau)-\Phi(g(\tau),\tau)|.
\end{eqnarray}
The first term of the RHS tends uniformly to zero since (as is easily seen) 
\begin{eqnarray*}
\|\Phi_k-\Phi\|_{C^0(\mathbb{P}^1\times \overline{\mathbb{D}})}\leq \|\phi_k-\phi\|_{C^0(\mathbb{P}^1\times \partial \mathbb{D})}\leq \|\phi_k-\phi\|_{C^2(\mathbb{P}^1\times \partial \mathbb{D})}\leq 1/k.
\end{eqnarray*} 
The second term of the RHS of (\ref{eq:harmdiff}) tends to zero uniformly on compacts since $g_k$ tends to $g$ uniformly on compacts and $\Phi$ is $C^1.$ Thus $h_k$ tends to $h$ uniformly on compacts which implies that $h$ is harmonic, and so the graph of $g$ is a harmonic disc of $\Phi.$ But since $U'$ was compactly supported in $U\cap ((\mathbb{P}^1\setminus \{0\})\times \mathbb{D}^{\times})$ the graph of $g$ must intersect $U$ which by Theorem \ref{thm:partialsmooth} is a contradiction. This concludes the proof. 
\end{proof}

\section{Partial and Almost Smoothness}\label{sec:ChenTian}

\begin{definition}\label{def:regularlocus}
Let $\Phi$ be upper semicontinuous on an open subset of $\mathbb P^1\times \overline{\mathbb D}$ with $\pi^*_{\mathbb P^1}\omega_{FS} + dd^c\Phi\ge 0$ and $(\pi^*_{\mathbb P^1}\omega_{FS} + dd^c\Phi)^2=0$.  The \emph{regular locus} $\mathcal R_{\Phi}$ is the set of all $(z,\tau)\in \mathbb P^1\times \overline{\mathbb D}$ near which $\Phi$ is smooth and $\pi_{\mathbb P^1}^*\omega_{FS} + dd^c\Phi(\cdot,\tau)$ is a K\"ahler form.
\end{definition}

In particular the above definition applies to the weak solution of the HMAE.  It is clear that $\mathcal R_{\Phi}$ is open, and inside it the kernel of $\pi^*_{\mathbb P^1}\omega_{FS} + dd^c \Phi$ defines a complex one-dimensional integrable distribution $\mathcal D_{\Phi}$ in $R_{\Phi}$.  By restriction this gives a foliation on any open subset $V\subset \mathcal R_{\phi}$ which, when necessary, we refer to as the \emph{Monge-Amp\`ere foliation} in $V$.

The following definition is a slight adaptation of those from \cite[Sec 1.3]{ChenTian}.



\begin{definition}
We say the weak solution $\Phi$ to the HMAE is \emph{partially smooth}\footnote{Prof Tian has informed us that the definition of partially smooth in  \cite[Def 1.3.1]{ChenTian} contains a typo, and should read that ``$R_{\phi}$ is saturated in $X\times (\Sigma\setminus\partial\Sigma)$''.  This is potentially weaker than what is written in \cite[Def 1.3.1]{ChenTian}, and is what we use here.} if it holds that (1) $\mathcal R_{\phi}\cap (\mathbb P^1\times \mathbb D)$ is foliated by harmonic discs and (2) $\mathcal R_{\phi}$ is dense in $\mathbb P^1\times \partial\mathbb D$ and (3) the fibrewise volume form  $\omega_{\Phi(\cdot,\tau)}$ which is defined on $R_{\phi}$ extends to a continuous $(1,1)$-form on $\mathbb P^1\times \mathbb D$.
\end{definition}

The following should be compared with \cite[Thm. 1.3.2]{ChenTian} which says weak solutions are always partially smooth.

\begin{corollary}
There exist $\phi$ such that the weak solution $\Phi$ to the HMAE with boundary data $\phi(z,\tau):=\phi(\rho(\tau)z)$ are not partially smooth.  
\end{corollary}
\begin{proof}
This is immediate, for our example from Theorem \ref{thm:partialsmooth} gives boundary data for which the solution to the HMAE is not partially smooth, since no harmonic disc intersects the open set $U$ which has non-empty intersection with the boundary.
\end{proof}

\begin{remark}
The example in Theorem \ref{thm:perturb} is in apparent contradiction with  \cite[Thm 1.3.4]{ChenTian}, since the solutions $\Phi'$ associated to the perturbed boundary data $\phi'$ cannot be ``almost smooth'' (for by \cite[Prop 2.3.1]{ChenTian} for any almost smooth solution, the set of harmonic discs is dense in $\mathbb P^1\times \mathbb D$).
\end{remark}

\medskip
\small{
\noindent {\sc Julius Ross,  DPMMS , University of Cambridge, UK. j.ross@dpmms.cam.ac.uk}\medskip

\noindent{\sc David Witt Nystr\"om, DPMMS,  University of Cambridge, UK. \newline d.wittnystrom@dpmms.cam.ac.uk, danspolitik@gmail.com}}


\begin{thebibliography}{99}

\bibitem{BedfordBurns} E Bedford and D Burns \emph{Holomorphic mapping of annuli in ${\bf C}^{n}$ and the associated extremal function.} Ann. Scuola Norm. Sup. Pisa Cl. Sci. (4) 6 (1979), no. 3, 381--414.

\bibitem{Bermanample} R Berman \emph{ Bergman kernels and equilibrium measures for ample line bundles} (2007) Preprint arXiv:0704.1640. 

\bibitem{Bermanreg} R Berman \emph{On the optimal regularity of weak geodesics in the space of metrics on a polarized manifold} (2014) Preprint arXiv:1405.6482.
 
\bibitem{BoBerman} R Berman and B Berndtsson \emph{ Convexity of the K-energy on the space of K\"ahler metrics} (2014) Preprint arXiv:1405.0401

\bibitem{Blocki} Z  B\l ocki \emph{On geodesics in the space of Kahler metrics} Advances in Geometric Analysis, ed. S. Janeczko et al., Advanced Lectures in Mathematics 21, pp. 3-20, International Press, 2012. 

\bibitem{Chen} X X Chen \emph{The Space of K\"ahler Metrics}     J. Differential Geom.    Volume 56, Number 2 (2000), 189--234.

\bibitem{ChenTian} X X Chen and G Tian \emph{Geometry of K\"ahler metrics and foliations by holomorphic discs}. Publ. Math. Inst. Hautes \'Etudes Sci. No. 107 (2008), 1--107.


\bibitem{Caffarelli} L.A. Caffarelli and N. M. Rivi\`ere \emph{Smoothness and analyticity of free boundaries in variational inequalities} Annali della Scuola Normale Superiore di Pisa, Classe di Scienze 4em s\'erie, tome 3, no 2 (1976), p. 289--310.

\bibitem{CaffarelliK} L.A. Caffarelli  and D. Kinderlehrer  \emph{Potential methods in variational inequalities} J. Analyse Math. 37 (1980), 285--295.

\bibitem{DarvasLempert} T Darvas and L Lempert \emph{Weak geodesics in the space of K\"ahler metrics} Math. Res. Lett. 19 (2012), no. 5, 1127--1135.

\bibitem{Darvas} T Darvas \emph{Morse theory and geodesics in the space of K\"ahler metrics}  Proc. Amer. Math. Soc. 142 (2014), no. 8, 2775--2782.
 


\bibitem{Donaldson}  S K Donaldson \emph{Holomorphic discs and the complex Monge-Amp\`ere equation}.  J. Symplectic Geom. 1 (2002), no. 2, 171--196. 

\bibitem{DonaldsonNahm} S K Donaldson \emph{Nahm's equations and free-boundary problems}. The many facets of geometry, 71--91, Oxford Univ. Press, Oxford, 2010. 

\bibitem{Guedjbook} Vincent Guedj (Editor) \emph{Complex Monge-Amp\'ere Equations and Geodesics in the Space of K\"ahler Metrics}  Lecture Notes in Mathematics 2038, Springer 2012


\bibitem{Gustafsson} B. Gustafsson and A. Vasil'ev \emph{Conformal and potential analysis in Hele-Shaw cells} Advances in Mathematical Fluid Fluid Mechanics. Birkh\"auser Verlag, Basel, 2006.

\bibitem{Hedenmalm} H. Hedenmalm and S. Shimorin \emph{Hele-Shaw flow on hyperbolic surfaces} J. Math. Pures Appl. (9) 81 (2002), no. 3, 187--222.

\bibitem{Kiselman} C Kiselman \emph{The partial Legendre transformation for plurisubharmonic functions} Invent. Math. 49 (1978), 137--148.

\bibitem{LempertVivas} L Lempert and L Vivas \emph{Geodesics in the space of K\"ahler metrics} Duke Math. J. 162 (2013), no. 7, 1369--1381. 

\bibitem{Rockafellar} R. T. Rockafellar \emph{Convex analysis} Princeton Mathematical Series, No. 28, Princeton University Press, Princeton, N.J. 1970.

\bibitem{RossNystromHeleShaw} J Ross and D Witt Nystr\"om \emph{ The Hele-Shaw flow and moduli of holomorphic discs} (2012) Preprint arXiv:1212.2337.


 \bibitem{Semmes} S Semmes \emph{Complex Monge-Amp\`ere and symplectic manifolds} Amer. J. Math. 114 (1992), no. 3, 495-550.


 \end{thebibliography}
\end{document}